\documentclass[10pt]{amsart}

\usepackage{latexsym, amsmath, amssymb, longtable, booktabs,amscd,microtype,booktabs}
\usepackage[square]{natbib}
\usepackage{hyperref}
\usepackage{xcolor}
\bibpunct{[}{]}{,}{n}{}{;} 
\numberwithin{equation}{section}
\newcommand\C{\mathbb{C}}

\newcommand\N{\mathbb{N}}

\newcommand\sE{{\mathcal{E}}}
\newcommand\tE{{\mathbb{E}}}

\newcommand\HH{{\mathrm H}}

\newcommand\sP{{\mathbb{P}}}
\newcommand{\Q}{\mathbb{Q}}
\newcommand{\tH}{\mathbb{H}}
\newcommand{\Z}{\mathbb{Z}}
\newcommand{\R}{\mathbb{R}}

\newcommand\SL{{\mathrm {SL}}}

\newtheorem{thm}{Theorem}
\newtheorem{theorem}[thm]{Theorem}
\newtheorem{cor}[thm]{Corollary}

\newtheorem{prop}[thm]{Proposition}
\newtheorem{lemma}[thm]{Lemma}
\theoremstyle{definition}
\newtheorem{definition}[thm]{Definition}
\theoremstyle{remark}
\newtheorem{remark}[thm]{Remark}

\theoremstyle{definition}

\theoremstyle{remark}

\theoremstyle{remark}

\makeatletter
\def\imod#1{\allowbreak\mkern10mu({\operator@font mod}\,\,#1)}
\makeatother

\begin{document}
\title{The Eisenstein and winding elements of modular symbols for odd square-free level}

\author{SriLakshmi Krishnamoorty}
\address{Indian Institute of Science Education and Research, Thiruvananthapuram, India}
\email{srilakshmi@iisertvm.ac.in}

\begin{abstract}
We explicitly write down the {\it Eisenstein elements} inside the space of modular symbols for  
Eisenstein series with integer coefficients for the congruence subgroups $\Gamma_0(N)$ with $N$ odd square-free. 
We also compute the { \it winding elements} explicitly 
for these congruence subgroups.  Our results are explicit versions of the Manin-Drinfeld Theorem [Thm.\,\ref{ManinDrinfeld}]. These results are the generalization of the paper \cite{MR3463038} results to odd square-free level.

\end{abstract}
\maketitle
\section{Introduction}
\label{Introduction}
Mazur gave a list of all 
possible torsion subgroups of elliptic curves [cf. Thm. 8, \cite{MR488287}].
 Merel 
wrote down explicit expression of the Eisenstein elements for the congruence subgroups $\Gamma_0(p)$ for any odd prime $p$ \cite{MR1405312}.
He used this to give an uniform upper bound of the torsion 
points of elliptic curves over any number fields in terms of extension degrees of these number fields
 \cite{MR1369424}. 
The explicit expressions of winding elements for $\Gamma_0(p)$ \cite{MR1405312} are used by Calegari and Emerton to 
study the ramifications of Hecke algebras at the Eisenstein primes \cite{MR2129709}. 
 Several authors investigated the arithmetic invariants of the elliptic curves 
over number fields using modular symbols.  

In this paper, we give an ``explicit version" of the proof 
of the Manin-Drinfeld theorem [Thm.\,\ref{ManinDrinfeld}] for the special case of the image in $\HH_1(X_0(N), \R)$ of the path in $\HH_1(X_0(N), \partial(X_0(N)), \Z)$ joining $0$  and $i\infty.$\\ 
For $m \mid N,$  consider the basis $\{ E_m \}$ of $E_2(\Gamma_0(N))$ [\S\,\ref{Eisensteinseries}] for which the constant term $a_0(E_m) \in \Z.$
\begin{definition}
[Eisenstein elements]
The intersection pairing $\circ$ \cite{MR1240651} induces a perfect, bilinear pairing 
\[
 \HH_1(X_0(N),\partial(X_0(N),\Z) \times \HH_1(Y_0(N),\Z) \rightarrow \Z. 
\]
Let $\pi_{E_m}:\HH_1(Y_0(N),\Z) \rightarrow \Z$
be the ``period'' homomorphism of $E_N$ [\S~\ref{Peiods}].
Since $\circ$ is a non-degenerate, 
there is an unique modular symbol $\sE_{E_m} \in \HH_1(X_0(N),\partial(X_0(N)),\Z)$ such that $\sE_{E_m} \circ c=\pi_{E_m}(c).$ 
This unique element $\sE_{E_m}$ is the {\it Eisenstein element} corresponding to the Eisenstein series 
$E_m$.
\end{definition}
\begin{definition} \label{winding element}
(Winding elements)
Let $\{0,\infty\}$ denote the projection 
of the path from $0$ to $\infty$ in $\tH \cup {\sP}^1(\Q)$ to $X_0(N)(\C)$. 
We have an isomorphism $\mathrm{H}_1(X_0(N),\Z)\otimes \R \simeq \mathrm{Hom}_{\C}(\mathrm{H}^0(X_0(N),\Omega^1),\C)$.
Let $e_{N} :  \mathrm{H}^0(X_0(N), \Omega^1) \rightarrow \C$ be given by  the homomorphism 
$\omega \mapsto - \int_0^{\infty} \omega$. The modular symbol $e_{N} \in\mathrm{H}_1(X_0(N),\R) $  is called the {\it winding element}.
\end{definition}
The winding elements are the elements of the space of modular symbols whose annihilators define ideals 
of the Hecke algebras with the $L$-functions of the corresponding quotients of the Jacobian non-zero.Since the algebraic part of the special values 
of $L$-function are obtained by integrating differential forms on these modular symbols, our explicit expression of the winding elements are
useful to understand the algebraic parts of the special values at $1$ of the $L$-functions of the quotient $J_0(N)$ \cite{amod}.
In this paper, we find an explicit expression of the winding elements and the Eisenstein elements
[Theorem \ref{Main-thm}].\\
Let $\overline{B}_1(x)$ be the first Bernoulli's polynomial of period $1$ defined by 
$\overline{B}_1(0)=0, \overline{B}_1(x)=x-\frac{1}{2}, \
\mathrm{if} \ x \in (0, 1).$\\
Let  $u, v \in \mathbb{Z}$  such that $\left( u,v \right) = 1$ and $v \geq 1$, we define the Dedekind sum 
by the formula :
\[
S(u,v)=\sum_{t=1}^{v-1} \overline{B}_1(\frac{tu}{v}) \overline{B}_1(\frac{u}{v}). 
\]
For any $k \in \mathbb{Z},$ let $\delta_k \in \{ 0, 1 \}$ such that $\delta_k \equiv k \pmod {2}.$ Let $s_k=(k+(\delta_k-1)N).$ Then $s_k$ is always an odd integer.  Let  $x$ be one of the prime divisor of $N.$ Any coset representative $g \in \sP^1(\Z \slash N \Z)$ can be written as either $(1, \pm t)$ or $(\pm t, 1).$ We observe that
 $(r-1, r+1) \sim (1, \pm t)$ or $(\pm t, 1)$ for some $r$  $\Leftrightarrow$ $t = \pm 1 \pm k x$ for some $k.$ 
 Choose integers 
$s, s'$ and $l, l'$ such that $l(s_kx+2)-2sN=1$ and $l's_kx-2s'\frac{N}{x}=1$.  Let $\gamma_1^{x,k}=\left(\begin{smallmatrix}
1+4sN&  -2l \\
- 4s(s_kx+2)N& 1+4sN\\
\end{smallmatrix}\right)$ 
and $\gamma^{x,k}_2=\left(\begin{smallmatrix}
1+4s'\frac{N}{x} &  -2l' \\
- 4s'(s_k)N& 1+4s'\frac{N}{x}\\
\end{smallmatrix}\right)$ 
For $y=1,2$, consider the integers 
\[
P_m(\gamma^{x,k}_{y})=sgn(t(\gamma^{x,k}_{y}))[2 (S(s(\gamma^{x,k}_y),|t(\gamma^{x,k}_y)|m)-S(s(\gamma^{x,k}_y), |t(\gamma^{x,k}_y)y))
\]
\[
-S(s(\gamma^{x,k}_y),\frac{|t(\gamma^{x,k}_yy)}{2}m)+S(s(\gamma^{x,k}_y), \frac{|t(\gamma^{x,k}_y)y}{2}),]
\]
where 
\[
 s(\gamma_1^{x,k})=1-4sN(1+s_kx), \ t(\gamma_1^{x,k})=-2(l-2s(s_kx+2)N), 
 s(\gamma_2^{x,k})=1-4s'N(s_k-\frac{1}{x}), \ t(\gamma_2^{x,k})=-2(l'-2s's_kN).
 \]
Let $\xi : \SL_2(\Z) \rightarrow \HH_1(X_0(N),\text{cusps},\Z)$ 
be the Manin map [~\ref{Mod}] and $F_m : \ 
  \sP^1(\Z \slash m \Z) \rightarrow \Z$ be defined by  
\begin{equation*}
F_{m}(g) =
\begin{cases}
\ 2(S(r,m)-2S(r,2m)) & \text{if $g=(r-1,r+1)$,} \\
\ P_m(\gamma_1^{x,k})-P_m(\gamma^{x,k}_2) &  \text{if $g=(1+kx,1)$ or $g=(-1-kx,1)$,} \\
-[P_m(\gamma_1^{x,k})-P_m(\gamma^{x,k}_2) ] &  \text{if $g=(1,1+kx)$ or $g=(1,-1-kx)$,} \\
\ 0 &  \text{if $g=(\pm 1,1).$}
\end{cases}
\end{equation*}

\begin{thm}
\label{Main-thm} Let $N$ be odd square-free level.
The modular symbol $$\sE_{E_m}= \sum_{g \in \sP^1(\Z \slash N \Z)} F_m(g) \xi(g) \in 
\HH_1(X_0(N), \partial(X_0(N), \Z) \  \mathrm{is} \ \mathrm{the} \ \mathrm{Eisenstein} \ \mathrm{element}$$
corresponding to the Eisenstein series $E_m \in E_2(\Gamma_0(N)).$
 $$\mathrm{The} \ \mathrm{element} \  e_{N}=\frac{1}{(1-N)}\sum_{v \in(\Z/N\Z)^*} F_{N}((1,v)) \{0,\frac{1}{v}\} \ \mathrm{is} \ \mathrm{the} \ \mathrm{winding} \
  \mathrm{element} \ \mathrm{for} \ \Gamma_0(N).$$\end{thm}
In \cite{debargha} and  \cite{MR3463038}, Eisenstein elements are described for $N=p^2,$ $N=pq$ respectively. In this article, we give an explicit description in terms of two matrices $\gamma_1^{x,k}$ and 
$\gamma_2^{x,k}$ for $N$ odd square-free level.
 Note that if $h=gcd(N-1,12)$ and $n=\frac{N-1}{h}$, then a multiple of winding element $ne_{N}$ belongs to 
$\HH_1(X_0(N), \Z)$.
Manin-Drinfeld proved that the modular symbol 
$\{0, \infty \} \in \HH_1(X_0(N), \Q).$  For the congruence subgroup $\Gamma_0(N)$ with $N$ square-free, we use the relative homology group $\HH_1(X_0(N), R \cup I, \Z)$ and $\HH_1(X_\Gamma-P_{-},P_{+},\Z).$ We intersect with the congruence 
subgroup $\Gamma(2)$ to ensure that the Manin maps become bijective (rather than only 
surjective). These 
relative homology groups are used in the study of modular symbol by the discovery of Merel. We follow his approach [cf.~\cite{MR1405312}, Prop. 11] to prove our results.\\
This paper is arranged as follows. In Section 2, we write down some preliminaries. In Section 3, we compute the coset representatives and cusps. In Section 4, we explicitly write down the even Eisenstein elements [Theorem \ref{psi-values}]. In Sections 5 and 6, we write the boundary of the Eisenstein series and Eisenstein elements. In Section 7, we porve our main theorem.
\section{Preliminaries}
For any natural number $M>4$, the congruence subgroup $ \{ \Gamma_0(M) = \left(\begin{smallmatrix}
a &  b\\
c & d\\
\end{smallmatrix}\right) \in \mathrm{SL}_2(\Z) /  M \mid c \}.$ This group $\Gamma_0(M)$
acts on  $\tH \cup  \sP^1(\Q) = \overline{\tH} .$  
Let  $Y_0(M)$ be the quotient space $\Gamma_0(M) \backslash \tH.$
These are non-compact Riemann surfaces and hence algebraic curves defined over $\C$. 
There are  models of these algebraic curve defined over $\Q$ and they parametrize elliptic curves with 
cyclic subgroups of order $M$.  Let $X_0(M)$ be the compactification of $Y_0(M).$ The set of cusps of $X_0(M)$ is given by $\Gamma_0(M) \backslash \sP^1(\Q).$ The modular curve $X_0(M) = Y_0(M) \cup \partial(X_0(M))$ is a 
smooth projective curve $X_0(N)$ defined over 
$\Q,$ we have $\Gamma_0(N ) \backslash \overline{\tH} \simeq X_0(N)(\C).$
We are interested to understand the $\Q$-structure of $X_0(N)$.  Let $\Gamma \subset \SL_2(\Z)$ be a congruence subgroup. The topological space 
$X_{\Gamma}(\C) =\Gamma \backslash \overline{\tH}$ has a natural structure of a smooth compact  Riemann surface.  The projection map  $\pi  : \overline{\tH} \rightarrow X_{\Gamma}(\C)$ is unramified outside the elliptic points and the 
set of cusps $\partial(X_{\Gamma})$. Both these sets are finite. 
\subsection{Classical modular symbols}
\label{Modularsymbols}
Recall the following fundamental theorem of Manin. 
\begin{thm}
\label{MaintheoremManin}
For $\alpha \in \overline{\tH}$, consider the map $c: \Gamma \rightarrow \HH_1(X_0(N), \Z)$
defined by $c(g)=\{\alpha, g \alpha\}.$
The map $c$ is a surjective group homomorphism  and does not depend on $\alpha$.
In fact, the kernel of this homomorphism is generated by the commutator, the elliptic elements, 
 and the parabolic elements of the congruence subgroup $\Gamma.$ In particular, $\{\alpha, g\alpha\}=0$ for all $\alpha \in \sP^1(\Q)$ 
and $g \in \Gamma$.

\end{thm}
\subsection{The Manin map}
  Let $T, S$ be the matrices $\left(\begin{smallmatrix}
1 & 1\\
0 & 1\\
\end{smallmatrix}\right),\left(\begin{smallmatrix}
0 & -1\\
1 & 0\\
\end{smallmatrix}\right)$
and  $R=ST$ be the matrix $\left(\begin{smallmatrix}
0 & -1\\
1 & 1\\
\end{smallmatrix}\right)$. 
The modular group $\SL_2(\Z)$ is generated by 
$S$ and $T$. 
\begin{thm}[Manin] \cite{MR0314846}
\label{Mod}
Let 
$\xi : \mathrm{SL}_2(\Z) \rightarrow \HH_1(X_0(N),\partial(X_0(N)),\Z)$
be the map that takes $g \in \mathrm{SL}_2(\Z)$ to 
the class in $\HH_1(X_0(N),\partial(X_0(N)), \Z)$ of the image in 
$X_0(N)$ of the geodesic in $\tH \cup \sP^1(\Q)$ joining $g.0$ and $g.\infty.$ 
Then $\xi$ is surjective and
$ \forall \ g \in \Gamma_0(N) \backslash \SL_2(\Z), \
\xi(g)+\xi(g S)=0$ and $\xi(g)+\xi(g R)+\xi(g R^2)=0.$
\end{thm}
\subsection{Manin-Drinfeld theorem}
We state the Manin-Drinfeld 
theorem. 
\begin{thm}[Manin-Drinfeld]
\label{ManinDrinfeld}
 \cite{MR0318157}
Let $\Gamma$ be a congruence subgroup and $\alpha, \beta \in \sP^1(\Q)$ be any two cusps. 
$$ \mathrm{The} \ \mathrm{path} \
\{ \alpha, \beta \} \in \HH_1(X_{\Gamma}, \Q).
$$
\end{thm}

As a corollary, we observe that $\{x, y\} \in \HH_1(X_{\Gamma}, \Q)$ if and if $m(\pi_{\Gamma}(x)-\pi_{\Gamma}(y))$ is a divisor of a function for some positive integer 
$m$. That is, 
the degree zero divisors supported on the cusps are of finite order in the divisor class 
group. Manin-Drinfeld proved it using the extended action of the usual Hecke operators.
In particular, it says that $\{0, \infty\} \in \HH_1(X_{\Gamma}, \Q).$
We have a short exact sequence,
$0 \rightarrow  \HH_1(X_0(N),\Z) \rightarrow \HH_1(X_0(N), \partial(X_0(N)),\Z)  \rightarrow {\Z}^{\partial(X_0(N))} \rightarrow \Z \rightarrow 0.
$
The first map is a canonical injection. The second map $
\delta'$ is a a boundary map which takes a geodesic,
joining the cusps $r$ and $s$ to 
the formal symbol $[r]- [s]$ and the third map is the sum of the coefficients. 
\subsection{Relative homology group $ \HH_1(X_0(N) - R \cup I, \partial(X_0(N)), \Z)$ and almost Eisenstein series}
\label{relative homology} Lret $\nu$ be the geodesic joining the elliptic points $i$ and $\rho=\frac{1+\sqrt{-3}}{2}$ of $X_0(N).$
Set $R = \pi(\SL_2(\Z)\rho)$ and
$I = \pi(\SL_2(\Z)i).$ A small checks shows that these two sets are disjoint.

 For $g \in  \SL_2(\Z)$, let
$[g]_*$ be the class of $\pi(g \nu)$ in the relative homology group $\HH_1(Y_0(N), R \cup I, \Z)$. Let $\rho^* = -\overline{\rho}$  be another
point on the boundary of the fundamental domain. The homology groups $\HH_1(Y_0(N), \Z)$ are subgroups of $\HH_1(Y_0(N), R \cup I, \Z)$. 
Suppose $z_0 \in  \tH$ be such that $|z_0| = 1$ and $\frac{-1}{2} < Re(z_0 ) < 1$. Let $\gamma$ be the union of the geodesic in $\tH \cup \sP^1(\Q)$
joining $0$ and $z_0$ and $z_0$ and $i \infty$. For $g \in \Gamma_0(N) \backslash \SL_2(\Z)$, let $[g]^*$ be the class of $\pi(g \gamma)$
in $\HH_1(X_0(N)- R \cup I, \partial(X_0(N)), \Z).$
We have an intersection pairing 
$$
\circ : \HH_1(X_0(N) - R \cup I, \partial(X_0(N)), \Z) \times \HH_1(Y_0(N), R \cup I, \Z) \rightarrow \Z \
\mathrm{and} \ \mathrm{the} \ \mathrm{following} \ \mathrm{results} \ \mathrm{of} \ \mathrm{Merel}.$$

\begin{prop}[Prop.\,1, Cor.\,1, \cite{MR1363498}].
\label{rho-value} For $g,h \in \SL_2(\Z),$ we have 
\begin{equation*}
 [g]^* \circ [h]_*= \begin{cases} $1$ & \text{if $\Gamma_0(N) g = \Gamma_0(N) h$} \\
$0$ & \text{otherwise.}
\end{cases}
\end{equation*}
\end{prop}
\begin{cor}
\label{rho-int}
 The homomorphism of groups $\Z^{\Gamma_0(N)\backslash \SL_2(\Z)} \rightarrow \HH_1(Y_0 (N), R \cup I, \Z)$ induced by the map
$$\xi_0 (\sum_g \mu_g g) = \sum_g \mu_g[g]_* \  \mathrm{is} \ \mathrm{an} \ \mathrm{isomorphism}.
$$
\end{cor}
We have a short exact sequence, $$
0  \rightarrow  \HH_1(X_0(N)-R \cup I ,\Z) \rightarrow \HH_1(X_0(N)-R \cup I , \partial(X_0(N)),\Z)  \rightarrow {\Z}^{\{\partial(X_0(N))\}} \rightarrow \Z \rightarrow 0.$$
The first map is a canonical injection. The second map $\delta$ is the boundary map  which takes a geodesic,
joining the cusps $r$ and $s$ to 
the formal symbol $[r]- [s]$ and the third map is the sum of the coefficients. Note that
$\delta'(\xi(g))=\delta([g]^*)$ for all $g \in \SL_2(\Z)$.
\begin{definition}[Almost Eisenstein elements]
For $m \mid N$, the differential form $E_m(z)dz$ is of first kind on the Riemann surface $Y_0(N) $. Since $\circ$ is a
non-degenerate bilinear pairing, there is an unique element $\sE'_{E_m} \in  \HH_1(X_0(N) - R \cup I, \partial(X_0(N)), \Z)$ such that
$\sE'_{E_m} \circ c = \pi_{E_m}(c)$ for all $c \in  \HH_1(Y_0 (N), R \cup I, \Z)$. We call $\sE'_{E_m}$ the {\it almost Eisenstein element} corresponding
to the Eisenstein series $E_m$.
\end{definition}
\subsection{Modular Curves with bijective manin maps}
The Riemann sphere or the projective complex plane $\sP^1(\C)$ is the only one simply connected (genus zero) compact Riemann surface up to conformal bijections. 
A theorem of Belyi states that every (compact, connected, non- singular) algebraic curve $X$ 
has a model defined over $\overline{\Q}$  if and only if it admits a map to $\sP^1(\C)$ branched over three points. Let $ \Gamma(2) =  \left(\begin{smallmatrix}
a &  b\\
c & d\\
\end{smallmatrix}\right) \in \mathrm{SL}_2(\Z) /   \left(\begin{smallmatrix}
a &  b\\
c & d\\
\end{smallmatrix}\right) \equiv  \left(\begin{smallmatrix}
1 &  0\\
0 & 1\\
\end{smallmatrix}\right) \pmod{2}\}$ and $X(2) =  \Gamma(2) \mod \overline{\tH}.$ Then
$\partial(X(2)) = \{ \Gamma(2) 0, \Gamma(2) 1,
 \Gamma(2) \infty \}.$ Hence, $X(2)$ is the simply 
 connected Riemann surface $\sP^1(\C)$ with the three marked points 
 $\Gamma(2) 0$,  $\Gamma(2) 1$ and 
 $\Gamma(2) \infty.$\\ 
Let  $\Gamma=\Gamma_0(N)
\cap \Gamma(2).$ Let $\pi_0:\Gamma \backslash  \tH \cup \sP^1(\Q) \rightarrow \Gamma(2) \backslash \tH \cup \sP^1(\Q) $ be the map $\pi_0(\Gamma z)=\Gamma(2)z.$ Let $P_{-}=\pi_0^{-1}(\Gamma(2)1),  \ P_{+} = [\pi_0^{-1}(\Gamma(2)0)] \cup [ \pi_0^{-1}(\Gamma(2)\infty)].$
The Riemann surface $X_{\Gamma}$ has boundary $P_{+} \cup P_{-}.$
\subsection{Relative homology group  $\HH_1(X_\Gamma-P_{-},P_{+},\Z)$ and even Eisenstein elements}
We study the relative homology groups $\HH_1(X_\Gamma-P_{-},P_{+},\Z)$ and 
$\HH_1(X_\Gamma-P_{+},P_{-},\Z)$. The intersection pairing is a non-degenerate bilinear pairing 
$\circ:\HH_1(X_\Gamma-P_{+},P_{-},\Z) \times \HH_1(X_\Gamma-P_{-},P_{+},\Z) \rightarrow \Z.$
For $g \in \Gamma \backslash \Gamma(2)$, let $[g]^0$ 
(respectively  $[g]_0$) be the image in $X_\Gamma$ of 
the geodesic in $\tH \cup \sP^1(\Q)$ joining $g0$ and $g\infty$ (respectively  $g1$ and $g(-1)$). 
We state two fundamental theorems.
\begin{thm}[\cite{MR1405312}]
\label{int-iso}
Let $\xi_0: {\Z}^{\Gamma \backslash \Gamma(2)} \rightarrow \HH_1(X_\Gamma-P_{+},P_{-},\Z), \ \xi^0:  {\Z}^{\Gamma \backslash \Gamma(2)} \rightarrow \HH_1(X_\Gamma-P_{-},P_{+},\Z)
$
be the maps which take $g \in \Gamma \backslash \Gamma(2)$ to the element $[g]_0$ and 
 $g \in \Gamma\backslash \Gamma(2)$ to the element $[g]^0$ respectively.\\
The homomorphisms $\xi_0$ and $\xi^0$ are isomorphisms.
\end{thm}
\begin{theorem}[\cite{MR1405312}]
\label{int-value}
For $g, g' \in \Gamma(2)$, we have 
$
 [g]_0 \circ [g']^0= \begin{cases}
1 & \text{if $\Gamma g = \Gamma g'$} \\ 
$0$ & \text{otherwise.}
 \end{cases}
$
\end{theorem}
\begin{definition}
\label{Eisensteinsec} (Even Eisenstein elements).
For $E_m \in \tE_{N}$, let 
$\lambda_{E_m}:X_0(N) \rightarrow \C$ be the rational function whose logarithmic differential is $2 \pi i E_m(z)dz=2\pi i \omega_{E_m}$. 
Consider the rational function $\lambda_{E_m,2}=\frac{(\lambda_{E_m} \circ \pi)^2}{\lambda_{E_m} \circ \pi'}$ on $X_\Gamma$. By Lemma 
~\ref{no-zero}, this function has no zeros and poles in $P_{-}$. Let $\kappa^*(\omega_{E_N})$ be the logarithmic differential of the function.
Let $\varphi_{E_m}(c)=\int_c \kappa^*(\omega_{E_m})$ be the corresponding ``period'' homomorphism $\HH_1(X_\Gamma-P_{+},P_{-},\Z) \rightarrow \Z$. By the non-degeneracy of the intersection pairing, there is a unique element $\sE_{E_m}^0 \in \HH_1(X_\Gamma-P_{-},P_{+},\Z)$ 
such that $\sE^0_{E_m} \circ c =\varphi_{E_m}(c)$ for all $c \in  \HH_1(X_\Gamma-P_{+},P_{-},\Z)$. The modular symbol $\sE_{E_m}^0$ is the {\it even Eisenstein element}
corresponding to the Eisenstein series $E_m$. 
\end{definition}
\subsection{Period Homomorphisms}
\label{Peiods} We state the period homomorphisms for the differential forms of third kind. Refer [\cite{MR553997}, p.\,10, \cite{MR1405312}, p.\,14] for some properties of the following period map $\pi_{E_m}.$
\begin{definition}[Period homomorphism]
For $E_m \in \tE_{N}$,  the differential forms $E_m(z)dz$ are of third kind on the Riemann surface
$X_0(N)$ but of first kind on the non-compact Riemann surface $Y_0(N)$.
 For any $z_0 \in \tH$ and $\gamma \in \Gamma_0(N)$, let $c(\gamma)$ 
be the class in $\HH_1(Y_0(N),\Z)$ of the image in $Y_0(N)$ of the geodesic in $\tH$ joining 
$z_0$ and $\gamma(z_0)$. 
This class is non-zero [Thm.~\ref{MaintheoremManin}] and is independent of the choice of $z_0 \in \tH.$ Let $\pi_{E_m}(\gamma)=\int_{c(\gamma)} E_m(z)dz.$ 
This map $\pi_{E_m}:\Gamma_0(N) \rightarrow \Z$
is the ``period'' homomorphism of $E_m$.
\end{definition}
\begin{prop}
\label{value-pi}
Let $\gamma=\left(\begin{smallmatrix}
a & b\\
c & d\\
\end{smallmatrix}\right)$ be an element of $\Gamma_0(N)$ and $\mu=gcd(m-1,12).$
\begin{enumerate}
 \item 
$\pi_{E_m}$ is a homomorphism $\Gamma_0(N) \rightarrow \Z$ and 
$\pi_{E_m}(\gamma)=\pi_{E_m}(\left(\begin{smallmatrix}
d & \frac{c}{m}\\
mb & a\\
\end{smallmatrix}\right)).$
\item 
The image of $\pi_{E_m}$ lies in $\mu \Z$ and
$\pi_{E_m}(\gamma) = 
\begin{cases}
\frac{a+d}{c} (m-1)+12 sgn(c)(S(d,|c|)-S(d,\frac{|c|}{m})) & \text{if $c \neq 0$,} \\
\frac{b}{d} (m-1) & \text{if $c=0$.}
\end{cases}$
\end{enumerate}
\end{prop}




\section{Coset representatives and Cusps}
We have a canonical bijection $\Gamma_0(N)\backslash \mathrm{SL}_2(\Z) \cong
{\sP}^1(\Z/N\Z)$ given 
by $\left(\begin{array}{cc}
a & b\\
c & d\\
\end{array}\right) \rightarrow (c,d).$\\ 
Let $\varOmega=  \cup_{m \mid N} M_m,$ where for any $m \mid N$ with $1 <m < N,$
 $$M_m = \{ \beta^m_l =\left(\begin{smallmatrix}
-1 & -l\\
m & l m -1\\
\end{smallmatrix}\right);  0 \leq l \leq \frac{N}{m}-1\}, \
M_N = \{ \alpha_N = \left( \begin{smallmatrix} 1 & 0 \\ 0 & 1 \end{smallmatrix} \right) \}, \
M_1 = \{ \alpha_k = \left(\begin{smallmatrix}
0 & -1\\
1 & {k}\\
\end{smallmatrix}\right); 0 \leq k \leq N-1\}.$$
\begin{lemma}
\label{explicit}
The set  $\varOmega$ is a complete set of  coset representatives of $\Gamma_0(N)\backslash \mathrm{SL}_2(\Z) \cong {\sP}^1(\Z/N\Z) $. 
\end{lemma}
\begin{proof}
Let $m_1,m_2$ be two divisors of $N$ such that $\beta^{m_1}_{l_1}(\beta^{m_2}_{l_2})^{-1} \in \Gamma_0(N)$
for some $0 \leq l_1 \leq \frac{N}{m_1}-1$ and $0 \leq l_2 \leq \frac{N}{m_2}-1$. This implies that 
$m_1m_2 (l_2-l_1) - m_1+m_2 \equiv 0 \pmod N.$ Since $m_1, m_2$ are divisors of $N$, the above expression implies that
$m_1=m_2$. Therefore the above expression reduces to
$m_1^2 (l_2-l_1) \equiv 0 \pmod N,$
which further implies that  	
$m_1(l_2-l_1) \equiv 0 \pmod {\frac{N}{m_1}}.$
Since $(m_1,\frac{N}{m_1})=1$, we get $l_1=l_2.$  Since $ab^{-1} \not \in \Gamma_0(N)$ for $a \in M_1 \cup M_N$ and $b \in M_m$ and $\# \varOmega = |{\sP}^1(\Z/N\Z)| =  \sum_{m \mid N} \frac{N}{m},$
the result follows.
\end{proof}
There are $2^{r}$ cusps of $X_0(N),$ where $r$ is the number of primes dividing
the level $N.$ They are explicitly given in the following lemma.
\begin{lemma}
\label{cuspspq}
The cusps $\partial(X_0(N))$ can be identified with the set $\{ 0, \infty\} \cup  
\{ \frac{1}{m}, m \mid N \}$. 
\end{lemma}
\begin{proof}
If $\frac{a}{c}$ and $\frac{a^{\prime}}{c^{\prime}}$ are in $\sP^1(\Q)$, then 
 $\Gamma_0(N) \frac{a}{c} = \Gamma_0(N) \frac{a^{\prime}}{c^{\prime}}$  
$\Longleftrightarrow$  $\left(\begin{array}{c} a y \\ c \\ \end{array}\right)  \equiv \left(\begin{array}{c} a^{\prime} +
yc^{\prime} \\ c^{\prime}x\\ \end{array}\right)$  $\pmod{N}$, for some $x$ and $y$ such that $\mathrm{gcd}(x,N)=1$
[\cite{MR2112196}, p. 99]. This implies that  
$\Gamma_0(N)0$, $\Gamma_0(N)\infty$ and $\Gamma_0(N)\frac{1}{m}, ( m \mid N )$  
are disjoint.
\end{proof}
We list the rational numbers $\alpha (0)$ and  $\alpha( \infty)$ with 
$\alpha \in \varOmega$ as equivalence classes of cusps as follows:
\[
\begin{tabular}{|l|l|l|}
\hline
$0$ & $\frac{1}{m}, \ m \mid N$ \\ \hline
$\frac{-l}{lm-1}$, $(lm-1,\frac{N}{m})=1$ &   $\frac{-1}{k}$, $(k,m) >1$ \\ \hline
$\frac{-t}{t\frac{N}{m}-1}$, $(t\frac{N}{m}-1,m)=1$ & $\frac{-a}{a\frac{N}{m}-1}$, $({a\frac{N}{m}-1},m) >1$ \\ \hline
\end{tabular}.
\]

Choose $a$ and $b$ be two unique integers such
that $a\frac{N}{m} + bm \equiv  1 \pmod{N}$ with $1 \leq a \leq (\frac{N}{m} - 1)$ and $1 \leq b \leq (m - 1).$  
Let $\tilde{\varOmega} =  \cup_{m \mid N} \tilde{M}_m,$ where for any $m \mid N$ with $1 <m < N,$
 $$\tilde{M}_m = \{ \tilde{\beta}^m_l = \left(\begin{smallmatrix}
-1 & -l_m + \delta_{l_m} \frac{N} {m}\\
(N + m) & \ -1+( N + m) ( -l_m + \delta_{l_{m}} \frac{N} {m})
\end{smallmatrix}\right);  0 \leq l_{m} \leq \frac{N}{m}-1\},$$
$$ \tilde{M}_N = \{
\tilde{\alpha}_{N}=\left(\begin{smallmatrix}
N  &  N-1\\ 
N+1 &   N\\
\end{smallmatrix}\right) \} \ \tilde{M}_1 = \{ \tilde{\alpha}_k =  \left(\begin{smallmatrix}
s_k N^2 & s_k N -1\\
s_k N+1 & s_k\\
\end{smallmatrix}\right); 0 \leq k \leq N-1\}.$$
%
\begin{lemma} \label{coset}
The set $\tilde{\varOmega}$
is a complete set of coset representatives of
 $\Gamma_0(N) \backslash  \mathrm{SL}_2(\Z) \cong {\sP}^1(\Z/N\Z)$ and
  $\Gamma \backslash \Gamma(2) \cong  \Gamma_0(N) \backslash  \mathrm{SL}_2(\Z).$
 \end{lemma}
\begin{proof}
The orbits 
$\Gamma_0(N) A_1$, $\Gamma_0(N) A_2$ 
are not equal, for distinct $A_1,$ $A_2$ $\in$ $\tilde{\varOmega}.$ The lemma follows from\\$|\tilde{\varOmega}| = |\sP^1(\Z/N\Z)|= \sum_{m \mid N} m.$
The set $\tilde{\varOmega} \subseteq \Gamma(2)$ and $\Gamma_0(N) \alpha_i $ = $\Gamma_0(N) \tilde{\alpha}_i,$  $\Gamma_0(N) {\beta}^m_l$ = $\Gamma_0(N) \tilde{\beta}^m_l.$  
Hence there is a canonical bijection $\Gamma_0(N) \backslash  \mathrm{SL}_2(\Z) \cong \Gamma \backslash \Gamma(2).$
\end{proof}

\begin{lemma}\label{Pminus}

We can explicitly write the set $P_{-}$ is of the form $\Gamma \frac{x}{y}$ with $x$ and $y$ both odd.
We can also write $P_{-}=\{ \Gamma \frac{1}{m},  m \mid N\}.$ 
\end{lemma}
\begin{proof} We can write any element $a$ in
$P_{-}=\pi_0^{-1}(\Gamma(2)1)$ as $\Gamma \theta1$ for some 
$\theta \in  \tilde{\varOmega}.$ (Lemma~\ref{coset}). Let $\delta \mid N$, 
then $a = \Gamma \frac{u}{v\delta}$ with 
$\gcd(u, v\delta)=1$ and $\gcd(v\delta, \frac{N}{\delta})=1$.
Choose  an odd integer $m$ and an even integer $l$  such that $lu-mv\delta=1$. 
We have
 $\left(\begin{smallmatrix}
1 & 0 \\
\delta-1 & 1\\
\end{smallmatrix}\right)\left(\begin{smallmatrix}
1+c & -c \\
c& 1-c\\
\end{smallmatrix}\right)1=\frac{1}{\delta}$ and $\left(\begin{smallmatrix}
-m & u+m\\
-l & l+v\delta\\
\end{smallmatrix}\right)1=\frac{u}{v\delta}.$ Hence 
 $A(\frac{u}{v\delta})=\frac{1}{\delta},$ where $A=\left(\begin{smallmatrix}
1 & 0 \\
\delta-1 & 1\\
\end{smallmatrix}\right)\left(\begin{smallmatrix}
1+c & -c \\
c& 1-c\\
\end{smallmatrix}\right)\left(\begin{smallmatrix}
 l+v\delta & -m-u\\
l & -m\\
\end{smallmatrix}\right).$ This matrix
$A \in \Gamma$ if and only if $c v\delta \equiv l' \pmod {\frac{N}{\delta}}$. Since 
$\gcd(v \delta, \frac{N}{\delta}) =1,$ there is always such $c$.
Hence, the set $P_{-}$ consists of $2^r$ elements $\{ \Gamma \frac{1}{m},  m \mid N\}.$ 
\end{proof}
\section{Even Eisenstein elements}
\label{Even}
We construct differential forms of first kind on the ambient 
Riemann surface $X_{\Gamma}-P_{+}$ by using the following lemma. 
\begin{lemma}
\label{no-zero}
 Let $f: X_0(N) \rightarrow \C$ be a rational function. The divisors of $\kappa(f)$ 
are supported on  $P_{+}$.
\end{lemma}
\begin{proof}
Let $f$ be a meromorphic function on  $X_0(N)$. Then $ f = \frac{g}{h}$ with $g$ and $h$ holomorphic function on $X_0(N)$. 
Every element of $P_{-}$ is of the form $\Gamma \frac{1}{m}$ with $m \mid N$ [Lemma \ref{Pminus}]. 
Every holomorphic map on Riemann surface locally looks like 
$z \rightarrow z^n$ [p.\,44, \cite{MR1326604}].

Consider the morphism   $\pi'$ and the point $\Gamma\frac{1}{m}$  with $m \mid N.$  
The local coordinates at the points 
$\Gamma_0(N)0,\Gamma_0(N)\infty$ and $\Gamma_0(N)\frac{1}{m}, \ 1 < m < N, m \mid N$ 
are given by $q_0(z)=e^{2\pi i \frac{1}{-Nz}}, q_\infty(z)=e^{2\pi i z}$ and $q_\frac{1}{m}(z)= e^{2\pi i \frac{z}{\frac{N}{m}(-mz+1)}}$
respectively. For $X_\Gamma$, the local coordinates around the points of $P_{-}$ are given by
$q_1(z)=e^{2\pi i \frac{1}{2N(-z+1)}}$, $q_{\frac{1}{N}}(z)=e^{2\pi i \frac{z}{2(-Nz+1)}}$, $q_\frac{1}{m}(z)=e^{2\pi i \frac{z}{2\frac{N}{m}(-mz+1)}}, \ 1< m < N, \ m \mid N.$
Around the point $\Gamma 1$ and $\Gamma \frac{1}{N}$, we have
$q_0 \circ \pi=q_1^2, \ q_0 \circ \pi'=q_1^4,$
 $q_{\frac{1}{N}} \circ \pi=q_{\frac{1}{N}}^2,$ and $q_ {\frac{1}{N}}\circ \pi'=q_{\frac{1}{N}}^4.$
Let $y=\frac{1}{m}$ with $1 < m < N, m \mid N.$ 

Consider the map $\pi$ and $t=\left(\begin{smallmatrix}
1& 0 \\
-m & 1\\
\end{smallmatrix}\right)$ is a matrix such that $t(y)=i \infty$ and $e=\frac{N}{m}$. 
The local coordinate around the point $\Gamma \frac{1}{m}$ is $z \rightarrow e^{2 \pi i \frac{t(z)}{2e}}$ and the map $\pi$ takes it to $e^{2 \pi i \frac{t(z)}{e}}$. In this coordinate chart, the map $\pi$ is given by $z \rightarrow z^2$ and  $\pi'$ is given by $z \rightarrow z^4.$ Thus the function $\frac{(f \circ \pi)^2}{f \circ \pi'}$ has no zero or pole on $P_{-}$.
\end{proof}
For $E_m \in \tE_{N}$, define a function $F_{m} : \sP^1(\Z/N\Z) \rightarrow \Z$ by
\[
F_{m}(g) =\varphi_{E_m}(\xi_0(g)) =\int_{g(1)}^{g(-1)}[2E_m(z) - E_m(\frac{z+1}{2})]dz.
\]
For any $\gamma=\left(\begin{smallmatrix}
a &  b \\
c & d\\
\end{smallmatrix}\right) \in \Gamma(2), \ h\gamma h^{-1}=\left(\begin{smallmatrix}
a+c &  \frac{b+d-a-c}{2} \\
2c & d-c\\
\end{smallmatrix}\right) \in \SL_2(\Z),$ where $h=\left(\begin{smallmatrix}
1 & 1\\
0 & 2\\
\end{smallmatrix}\right)$
and
for any matrix For $\gamma=\left(\begin{smallmatrix}
a & b \\
c & d\\
\end{smallmatrix}\right) \in \Gamma,$ with $c \neq 0,$ the rational number $P_m(\gamma)=\frac{2 \pi_{E_m}(\gamma)-\pi_{E_m}(h\gamma h^{-1})}{12}$ is given by \[
P_m(\gamma)=sgn(t(\gamma))[2 (S(s(\gamma),|t(\gamma)|N)-S(s(\gamma), |t(\gamma)|))
-S(s(\gamma),|\frac{t(\gamma)}{2}|N)+S(s(\gamma), \frac{|t(\gamma)|}{2})],
\]
where $t(\gamma)=b+d-a-c$ and $s(\gamma)=a+c.$ In particular, $P_m(\gamma) \in \Z$  [ \cite{MR3463038}, Remark $27,$ Lemma $28$ ].

\begin{prop}
\label{exceptional}
\begin{equation*}
F_{m}(g) =
\begin{cases}
12(S(r,m)-2S(r,2m)) & \text{if $g=(r-1,r+1)$,} \\
6(P_m(\gamma_1^{x,k})-P_m(\gamma^{x,k}_2)) &  \text{if $g=(1+kx,1)$ or $g=(-1-kx,1)$,} \\
-6(P_m(\gamma_1^{x,k})-P_m(\gamma^{x,k}_2)) &  \text{if $g=(1,-1-kx)$ or $g=(1,1+kx)$,} \\
0 &  \text{if $g=(\pm 1,1)$.} 
\end{cases}
\end{equation*}
\end{prop}

\begin{proof}
If $g = (r - 1, r + 1)$ and $E_m \in \tE_{N}$, we get [\cite{MR1405312}, p.\,18]
\[
F_{m}(g) = \varphi_{E_m}(\xi_0(g)) = 12(S(r, m ) - 2S(r, 2m )).
\]
The differential form $k^*(\omega_{E_m})$ is of first kind on the Riemann surface $X_{\Gamma}-P_{+}$.
Since all the Fourier coefficients of the Eisenstein series are real valued, so an argument similar to [\cite{MR1405312}, p.\,19] shows that 
$F_{m}(s_kx+1,1)=F_{m}(-s_kx-1,1)$. The path 
$
\{\frac{1}{s_kx+2},-\frac{1}{s_kx+2}\}=\{\frac{1}{s_kx+2},\frac{1}{s_kx}\}
+\{\frac{1}{s_kx},\frac{-1}{s_kx}\}
+\{\frac{-1}{s_kx},\frac{-1}{s_kx+2}\}.
$
The rational number $\frac{1}{s_kx}$ correspond to a point of $P_{-}$ in the Riemann surface $X_{\Gamma}$. 
The differential form $k^*{\omega_{E_m}}$ has no zeros and poles on $P_{-}$. 
We conclude that 
$$
\int_{\frac{1}{s_kx+2}}^{-\frac{1}{s_kx+2}} k^*(\omega_{E_m})= \int_{\frac{1}{s_kx+2}}^{\frac{1}{s_kx}} k^*(\omega_{E_m})
+\int_{\frac{1}{s_kx}}^{\frac{-1}{s_kx}} k^*(\omega_{E_m})
+\int_{\frac{-1}{s_kx}}^{\frac{-1}{s_kx+2}} k^*(\omega_{E_m})=2 F_m(s_kx+1,1)+\int_{\frac{1}{s_kx}}^{\frac{-1}{s_kx}} k^*(\omega_{E_m}). 
$$
Let $\gamma_1^{x,k}$ and $\gamma^{x,k}_2$ be two matrices in  $\Gamma$ such that  $\gamma_1^{x,k}(\frac{1}{s_kx+2})=-\frac{1}{s_kx+2}$ and $\gamma^{x,k}_2(\frac{1}{s_kx})=-\frac{1}{s_kx}$. 
Hence $2 F_m(s_kx+1,1)=\int_{\frac{1}{s_kx+2}}^{\gamma_1^{x,k}(\frac{1}{s_kx+2})} k^*(\omega_{E_N})
-\int_{\frac{1}{s_kx}}^{\gamma^{x,k}_2(\frac{1}{s_kx})} k^*(\omega_{E_m}).$ 

We now prove  that the $\int_{\frac{1}{s_kx}}^{\gamma^{x,k}_2(\frac{1}{s_kx})} k^*(\omega_{E_m})$ 
is independent of the choice of the matrices $\gamma^{x,k}_2 \in \Gamma$
that take $\frac{1}{s_kx}$ to $-\frac{1}{s_kx}$.  For, $\gamma^{x,k}_2$ and 
${\gamma'}^{x,k}_2$ are two matrices such that  $\gamma^{x,k}_2(\frac{1}{s_kx})=
{\gamma'}^{x,k}_2(\frac{1}{s_kx})=-\frac{1}{s_kx}$.  Since  $\gamma^{x,k}_2 \in \Gamma$, the integral $\varphi_{E_m}(\gamma^{x,k}_2)=\int_{\frac{1}{s_kx}}^{\gamma^{x,k}_2(\frac{1}{s_kx})} k^*(\omega_{E_m})$ is independent of the choice of any point in $\tH \cup \{-1\}$,  hence by 
replacing  $\frac{1}{s_kx}$ with $(\gamma^{x,k}_2)^{-1} ({\gamma'}^{x,k}_2)\frac{1}{s_kx}$, 
we get the above integral is same as $\int_{\frac{1}{s_kx}}^{{\gamma'}^{x,k}_2(\frac{1}{s_kx})} k^*(\omega_{E_m})$ and the integral is independent of the choice of the matrices. 
Similarly, we can prove that $\int_{\frac{1}{s_kx+2}}^{\gamma^{x,k}(\frac{1}{s_kx+2})}k^*(\omega_{E_m})$ 
is also independent of the choice of the matrices that take $\frac{1}{s_kx+2}$ 
to $-\frac{1}{s_kx+2}.$
The above calculation shows that
\[
2 \pi_{E_m}(\gamma_1^{x,k})-\pi_{E_m}(h \gamma_1^{x,k} h^{-1})=2 F_m(s_kx+1,1)+2 \pi_{E_m}(\gamma^{x,k}_2)
-\pi_{E_m}(h \gamma^{x,k}_2h^{-1}).
\]
Hence, we get $$F_{m}(s_kx+1,1)=\frac{2 \pi_{E_m}(\gamma_1^{x,k})-\pi_{E_m}(h \gamma_1^{x,k} h^{-1})-2 \pi_E(\gamma^{x,k}_2)+\pi_E(h \gamma^{x,k}_2 h^{-1})}{2}=6(P_m(\gamma^{x,k})-P_m(\gamma^{x,k}_2)).$$
We also have $F_{m}((1+s_kx,1))=-F_{m}((1,-1-s_kx))$, the proposition follows.
\end{proof}

\begin{thm}
\label{psi-values}
 For $E_m \in E_2(\Gamma_0(N))$,  the even Eisenstein elements $\sE_{E_m}^0$ of $ \HH_1(X_{\Gamma}-P_{-} , P_{+},\Z)$ is given by 
$\sE_{E_m}^0=\sum_{g \in \sP^1(\Z/N\Z)} F_{m}(g) \xi^0(g).$ \end{thm}
\begin{proof}
Let the even Eisenstein element be $ \sE_{E_m}^0 = \sum_{g \in \sP^1(\Z/N\Z)} H_{E_m}(g) \xi^0(g)$
for some $H_{E_m}(g).$ Since by thoerem \ref{int-value},
$H_{E_m}(g)=\sum_{g \in \sP^1(\Z/N\Z)} H_{E_m}(g) \xi^0(g) \circ \xi_0(g)= \sE^0_{E_m} \circ \xi_0(g)= \varphi_{E_m}(\xi_0(g)) = F_{m}(g)$.
\end{proof}
\section{Boundary of the Eisenstein series for \texorpdfstring{$\Gamma_0(N)$}{X}}
\label{Eisensteinseries}
Let $\sigma_1(n)$ denote 
the sum of the positive divisors of $n$. Let $E'_2(z)=1-24(\sum_n \sigma_1(n) e^{2 \pi i n z})$
and $\Delta$ be the Ramanujan's cusp form of weight $12$.
For all $N \in \N$, the function 
$z \rightarrow \frac{\Delta(Nz)}{\Delta(z)}$ is a function on $\tH$ invariant under $\Gamma_0(N)$.  
The logarithmic differential of this function is 
$2 \pi i E_N(z) dz$ and  $E_N$ is a modular form of weight  two for 
$\Gamma_0(N)$ with constant 
term $N-1$. The differential form $E_N(z) dz$ is a differential form of third kind on $X_0(N)$. The 
periods [\S\,\ref{Peiods}] of these differential forms are in $\Z.$ 
By  [\citep{MR2112196}, Thm.\, 4.6.2], the set $\tE_{N}=\{E_m, m > 1, m \mid N \}$ 
is a basis of $E_2(\Gamma_0(N)).$ 
Let $\mathrm{Div}^{0}(X_0(N), \partial(X_0(N)),\Z)$ be the group of degree zero divisors supported on cusps.
For all cusps $y$, let $e_{\Gamma_0(N)}(y)$ denote the ramification index of $y$ over $\SL_2(\Z) \backslash \tH \cup \sP^1(\Q)$ and 
 $r_{\Gamma_0(N)}(y)=e_{\Gamma_0(N)} (y)a_0(E[y]).$
By [\cite{MR670070}, p.\, 23], there is a canonical isomorphism $\delta : E_2(\Gamma_0(N)) \rightarrow \mathrm{Div}^0(X_0 (N), \partial(X_0(N)), \Z)$ given by \begin{equation}
\label{divisor}
\delta(E)=\sum_{y \in \Gamma_0(N) \backslash \sP^1(\Q)} r_{\Gamma_0(N)} (y)[y].
\end{equation} By \cite{MR800251}, we see that $e_{\Gamma_0(N)}(y) =
\begin{cases}
\frac{N}{m}& \text{if $x=\frac{1}{m}$} \\
$1$ & \text{if $y= \infty$}\\
$N$ & \text{if $y = 0.$} 
\end{cases} 
$
\begin{equation*}
\mathrm{Since} \ \displaystyle \sum_{x \in \partial(X_0(N))}e_{\Gamma_0(N)}(x)a_0(E[x])=0, \ \mathrm{we} \ \mathrm{get} \ \delta(E)=  a_0(E) (\{ \infty \} - \{0 \}) + \sum_{1<m <N, m \mid N}\frac{N}{m} a_0(E[\frac{1}{m}]).
\end{equation*}
\section{Boudaries of the Eisenstein elements} The level $N$ is square-free. Hence for $m \mid N,$
$N$ with $1<m<N$, there exists $a(m), b(m)$ are two unique integers such that
$a(m) \frac{N}{m}+ b(m)  m  \equiv 1 \pmod{N}$ with $1 \leq a(m) \leq m - 1$ and $1 \leq  b(m) \leq \frac{N}{m} - 1.$

\begin{lemma}
\label{twist}
For all $k$ with $1 \leq k \leq \frac{N}{m} - 1,$ we can choose an integer $s(k) \in (\Z/\frac{N}{m}\Z)$ such that
$(km, -1) = (m, s(k)m -1)$
in $\sP^1(\Z/N\Z)$. The map $k \rightarrow s(k)$ is a bijection\\ 
$(\Z/ \frac{N}{m} \Z)^* \rightarrow
\phi ( \prod_{p \mid N} (\Z/p\Z) - \{b_m(p)\} )$, where $\phi : \prod_{p \mid N} (\Z/p\Z) \rightarrow
(\Z / \frac{N}{m} \Z)$ is the standard isomorphism and $b_m(p)$ are chosen such that 
$ a_m(p) p  + b_m(p) m  =  1 \ \forall p \mid \frac{N}{m}$.
\end{lemma}
\begin{proof}
	 For all $k$ with $1 \leq k \leq (\frac{N}{m})^*$, let $k'$ be the inverse of $k$ in $(\Z/\frac{N}{m}\Z)^*$.  
	 
	By  Chinese remainder theorem, we choose an unique $x$ such that $x \equiv -1 \pmod m$ and 
 $x \equiv -k' \pmod {\frac{N}{m}}.$  It is possible to find such an $x.$ Observe
 $x$ is coprime to both $m$ and $\frac{N}{m}.$ We write $x = s(k)m -1$ for an unique $s(k)$ with 
 $0 \leq s(k) \leq \frac{N}{m} -1.$ Since
$\Gamma_0(N)\backslash \SL_2 (\Z) \cong \sP^1(\Z/N\Z)$, we deduce that $(km, -1) = (xkm, -x) = (-m, -x) = (m, x) = (m, s(k)m - 1)$
in $\sP^1(\Z/N\Z)$.

 Consider the map $\psi : (\Z/\frac{N}{m}\Z)^* \rightarrow (\Z/\frac{N}{m}\Z)$ given by $k \rightarrow s(k)$.  This map is one-one since if $s(y) = s(h)$ then $y \equiv h \pmod {\frac{N}{m}}.$ 
We have $\psi ( (\Z/\frac{N}{m}\Z)^* ) \subseteq \phi ( \prod_{p \mid N} (\Z/p\Z) - \{b^{m}(p)\} )$, where
$b_m(p)$ are chosen such that $ a_m(p) p  + b_m(p) m  =  1 \forall p \mid \frac{N}{m}$. 
For suppose $x \in (\Z/\frac{N}{m}\Z) -  \phi ( \prod_{p \mid N} (\Z/p\Z) - \{b_m(p)\} )$, then 
there exists a prime $p$ which divides $\frac{N}{m}$ such that $x \equiv \phi ( b_m(p)).$
Suppose $x \in \mathrm{Image}$ of $\psi$. We have $x = s(k)$ for some $k.$ 
Hence $xm-1$ is a unit mod $\frac{N}{m}$. Thus $xm-1$ is a unit mod $p$. 
Now look at $xm-1$ (mod $p$) = $b_m(p)m-1$ = 0 ( mod $p$).  Which is a contradiction. 
Hence the map $\psi$ is onto.
Hence, the map
\\ 
$(\Z/ \frac{N}{m} \Z)^* \rightarrow
\phi ( \prod_{p \mid N} (\Z/p\Z) - \{b_m(p)\} )$, $k \rightarrow s(k)$ is a bijection.
\end{proof}
\begin{prop}
\label{boundary}
 $\mathrm{The} \ \mathrm{boundary} \ \mathrm{of} X =\sum_{g \in \sP^1(\Z/N\Z)} F(g) [g]^*
\in \HH_1(X_0(N)-R \cup I, \partial(X_0(N)),\Z) \ \mathrm{is}$
$$
 \delta(X)= \sum_{m \mid N, 1 <m < N} A_m(X)([\frac{1}{m}] - [0] )+ C(X)( [\infty]-[0]), 
\ \mathrm{with}$$
$$A_m(X)=\sum_{l=0}^{\frac{N}{m}-1} [F(\beta^m_{l}) - F(\beta^m_{l} S)] \  \mathrm{and} \  \ C(X)=[F(0,1)-F(1,0)].
$$
\end{prop}
\begin{proof}
The proof follows along the same line as the proof of  [Proposition 32, \cite{MR3463038}]. 
\end{proof}
\begin{prop}
\label{boundaryeven}
 The boundary of any element $X =\sum_{g \in \sP^1(\Z/N\Z)} F(g) \xi^0(g) \in
 \HH_1(X_\Gamma-P_{-}, P_{+},\Z) \ \mathrm{is}$ $$\delta^0(X)=\sum_{m \mid N, 1 <m < N}  \tilde{A}_m(X)([\frac{1}{m}]-[0])+\tilde{C}(X)([\infty]-[0]), 
$$
$$\mathrm{where} \ \tilde{A}_m(X)=\sum_{l=0}^{\frac{N}{m}-1} [F(\tilde{\beta}^m_{l}) -[\sum_{l=1}^{{\frac{N}{m}-1}} F(\tilde{\alpha}_{lm})]- F( \beta_{a(m)}^{ \frac{N}{m}}),
\ \mathrm{and} \ \tilde{C}(X)=[F(0,1)-F(\tilde{\alpha}_{N})].$$ 
\end{prop}
\begin{proof}
 This is a straightforward calculation using the coset representatives of $\Gamma \backslash \Gamma(2)$ [cf.\, Lemma \ref{coset}].
\end{proof}
\begin{prop}
\label{boundact}
For $E \in \tE_{N}$, the boundary of the alomost Eisenstein elements $\sE'_E$ in $\HH_1(X_0(N)-R \cup I, \partial(X_0(N)),\Z)$ corresponding 
to the Eisenstein series $E$ is $-\delta(E)$ [\S  \ref{Eisensteinseries}].  
\end{prop}
\begin{proof}
For $E \in \tE_{N}$, let $\sE'_E = \sum_{g \in \sP^1 (\Z/N\Z)} G_E(g)[g]^*$ be the almost Eisenstein element. According to Proposition \ref{boundary}, we need to calculate $A_m(\sE'_E )$ and $C(\sE'_E ).$ For all $0 \leq l < (\frac{N}{m} - 1)$, $\beta^m_{l} T = \beta^m_{l+1}$ and 
$\beta_{l-1}^{\frac{N}{m}} T = \gamma \beta_0$ with 
$\gamma=\left(\begin{smallmatrix}
1+N & \frac{N}{m}\\
-\frac{N}{m}m^2  & 1-N\\
\end{smallmatrix}\right)$. We have an inclusion $\HH_1(Y_0(N), \Z) \rightarrow \HH_1(Y_0(N), R \cup I, \Z)$. Since 
$\{\rho^* , \gamma \rho^* \} =\{\beta_0\rho^* , \gamma \beta_0\rho^* \}=-\sum_{k=0}^{\frac{N}{m}-1} \{\beta_{l}^m \rho, \beta_{l}^m \rho^*\}$, 
we deduce that 
\[
\pi_E(\gamma)=\int_{z_0}^{\gamma z_0}E(z) dz=\sE'_E \circ \{z_0, \gamma z_0\}=-\sE'_E \circ (\sum_{k=0}^{\frac{N}{m}-1}  \{\beta_{l}^m \rho, \beta_{l}^m \rho^*\})
=-\sum_{k=0}^{\frac{N}{m}-1} \sE'_E \circ
 \{\beta_{l}^m \rho, \beta_{l}^m \rho^{*} \}. 
\]
Applying Cor.~\ref{rho-int}, we have
$\sum_{k=0}^{ \frac{N}{m} -1} \sE'_E \circ  \{\beta_{l}^m \rho, \beta_{l}^m \rho^{*} \}=
\sum^{\frac{N}{m}-1}_{k=0} [G_E(\beta_{l}^m)-G_E(\beta_{l}^m S)] 
=-A_m(\sE'_E).$ 
Hence, we prove that $A_m(\sE'_E )=-\pi_E(\gamma)$. 
We now calculate  $\pi_E(\gamma)$ and $\pi_E(\gamma_0 )$ using \cite{MR800251}. Recall,
$\frac{1}{m}$ is a cusp with $e_{\Gamma_0(N)}(\frac{1}{m})=\frac{N}{m}$. Consider 
the matrices $x=\left(\begin{smallmatrix}
1 & -\frac{N}{m}\\
-m  & 1+N\\
\end{smallmatrix}\right)$ and $y=\left(\begin{smallmatrix}
1 & -m\\
-\frac{N}{m}  & 1+N\\
\end{smallmatrix}\right)$ respectively. 
We have $x\left(\begin{smallmatrix}
1 & \frac{N}{m}\\
0  & 1\\
\end{smallmatrix}\right) x^{-1}= \gamma.$  Notice that $x(i \infty) = \Gamma_0(N) \frac{1}{m}.$ 
By [\cite{MR800251}, p. 524], we deduce that $\pi_E(\gamma) = e_{\Gamma_0(N)} ( \frac{m}{N}) a_0(E[\frac{1}{m}])$ and $\pi_{E_{N}}(\gamma_0)= e_{\Gamma_0(N)} (\frac{1}{m}) a_0(E[\frac{1}{m}])$. 
By Proposition~\ref{boundary}, the boundary of the almost Eisenstein element corresponding to an
Eisenstein series $E$ is
\[
\delta(\sE'_E)=\sum_{1<m<N, m \mid N} A(\sE'_E)[\frac{1}{m}]++C(\sE'_E)[\infty]-(A(\sE'_E)+C(\sE'_E))[0]
\]
with $A_m(\sE'_E)=\frac{N}{m}a_0(E[\frac{1}{m}])$ and $C(\sE'_E)=-[F(I)-F(S)]$. 
By Cor.~\ref{rho-int} again, we deduce that $F(I)-F(S)=
\int_{\rho}^{\rho^*}E(z)dz=-a_0(E).$ Hence \ref{Eisensteinseries}  $\delta(E) = \delta(\sE'_E).$ \end{proof}

Let $\beta$ and $h$ be the matrices $\left(\begin{smallmatrix}
1 & 2\\
0  & 1\\
\end{smallmatrix}\right)$ and $\left(\begin{smallmatrix}
1 & 1\\
0  & 2\\
\end{smallmatrix}\right)$
respectively. 
 The modular curve $X_0(N)$ has no 
 obvious morphism to $X(2).$ Hence, we consider the modular curve $X_{\Gamma}.$ There are two natural maps $\pi, \pi':\Gamma \backslash  \overline{\tH} \rightarrow \Gamma_0(N) \backslash \overline{\tH}$ 
be the maps $\pi(\Gamma z)=\Gamma_0(N)z$ and 
$\pi'(\Gamma z)=\Gamma_0(N)\frac{z+1}{2}$ respectively. 
For the modular curve $X_{\Gamma}$, we have a similar short exact sequence
\[
0  \rightarrow  \HH_1(X_{\Gamma}-P_{-} ,\Z) \rightarrow \HH_1(X_{\Gamma}-P_{-} , P_{+},\Z)  \xrightarrow{\delta^0} {\Z}^{P_{+}} \rightarrow  \Z \rightarrow 0.
\]
The boundary map $\delta^0$ takes a geodesic, joining the the point $r$ and $s$ of $P_{+}$ to the formal symbol $[r] - [s]$. Let 
$
 \pi_*:\HH_1(X_\Gamma-P_{-},P_{+},\Z) \rightarrow \HH_1(X_0(N)-R \cup I, \partial(X_0(N)),\Z)
$
be the isomorphism defined by $\pi_*(\xi_0(g))=[g]^*$ [\cite{MR1363498}, Cor.\,1]. It is easy to see that $\delta(\pi_*(X))=\delta^0(X)$ for all 
$X \in \HH_1(X_\Gamma-P_{-},P_{+},\Z).$
\begin{prop}
\label{bound}
For all $E \in \tE_{N}$, let $\sE_E^0 $ be the even Eisenstein element in $\HH_1(X_\Gamma-P_{-},P_{+},\Z)$ [\S ~\ref{Even}]. 
The boundary of the modular symbol $\pi_*(\sE^0_E)$ is $-6\delta(E)$.
\end{prop}
\begin{proof}
By Theorem \ref{psi-values}, suppose the even Eisenstein element $\sE_E^0$ in the relative homology
group\\$\HH_1(X_{\Gamma} - P_{-}, P_{+}, \Z)$ is
$\sE_E^0 =\sum_{g \in \sP^1(\Z/N\Z)} F_E(g)\xi_0 (g).$
According to Proposition \ref{boundaryeven}, we need to
calculate $\tilde{A}_m( \sE_E^0), \tilde{C}( \sE_E^0).$ 
For $0 \leq l_m < (\frac{N}{m} -2)$, we have 
$\tilde{\beta}_{l}^m \beta = \tilde{\beta}_{l+2}^m$ . A small check shows
that $\tilde{\beta}_{l-1}^m \beta = \tilde{\beta}_1$ and 
$\tilde{\beta}_{l-2}^{\frac{N}{m}} \beta = \gamma' \tilde{\beta}_0$ with
\[
\gamma'=\left(\begin{smallmatrix}
1+2 N(1+\frac{N}{m}) & 2\frac{N}{m} \\
-2\frac{N}{m}(m+N)^2 & 1-2 N(1+\frac{N}{m})\\
\end{smallmatrix}\right) \in \Gamma.
\]
In $\HH_1(X_{\Gamma} - P_{+}, P_{-} ,\Z)$, we have
\[
\{-1, \gamma'(-1)\} = \{\tilde{\beta}_0 (-1), \gamma' \tilde{\beta}_0(-1)\} = 
-\sum_{l_m=0}^{\frac{N}{m}-1}\{\tilde{\beta}_{l}^m(1), \tilde{\beta}_{l}^m (-1)\}
=\sum_{l_m=0}^{\frac{N}{m}-1}\{\tilde{\beta}_{l-1}^{\frac{N}{m}} (-1), 
\tilde{\beta}_{l}^m (1)\}.
\]
By the definition of the even Eisenstein elements, we conclude that
\[
\int_{z_0}^{\tilde{\beta} z_0} k^*(\omega_E )= \sE_E^0 \circ \{z_0, 
\gamma' z_0 \} = 
- \sE_E^0\circ (\sum^{\frac{N}{m}-1}_{{l}=0}
({\tilde{\beta}_{l}^m(1), \tilde{\beta}_{l}^m (-1)})
=-\sum^{\frac{N}{m}-1}_{{l}=0}\sE_E^0 \circ \{\tilde{\beta}_{l}^m(1),
 \tilde{\beta}_{l}^m (-1)\}.
\]
It is easy to see that $h AS B h^{-1} \in \SL_2(\Z)$ for all
 $A, B \in \Gamma(2)$.  Since $[\tilde{\alpha}_{y \frac{N}{m}} S] =
  [\tilde{\beta}_{s(y)}]$ in $\sP^1(\Z/N\Z)$, so
$\kappa'=\tilde{\alpha}_{y\frac{N}{m}} S(\tilde{\beta}_{s(y)}^m)^{-1} \in 
\Gamma_0(N)$ and $h \kappa'h^{-1}\in \Gamma_0(N)$. 
We deduce that the differential form
\[
k^*(\omega_E)=f(z)dz=[2 E(z)-\frac{1}{2} E(\frac{z+1}{2})] dz
\]
is invariant under $\kappa'$. Hence \begin{equation}
\label{valuered}
F_E(\tilde{\alpha}_{y\frac{N}{m}})=
\int_{\tilde{\alpha}_{y\frac{N}{m}}(1)}^{\tilde{\alpha}_{y\frac{N}{m}}(-1)}f(z)dz
=\int_{\tilde{\alpha}_{y\frac{N}{m}}S(-1)}^{\tilde{\alpha}_{kq}S(1)}f(z)dz=
-\int_{\tilde{\alpha}_{y\frac{N}{m}}S(1)}^{\tilde{\alpha}_{y\frac{N}{m}}S(-1)}
f(z)dz
\end{equation}
\[
=-\int_{\kappa'^{-1}\tilde{\alpha}_{y\frac{N}{m}}S(1)}^{\kappa'^{-1}\tilde{\alpha}_{y\frac{N}{m}}S(-1)}f(\kappa'z)d\kappa'z=-\int_{\gamma'_{s(y)}(1)}^{\gamma'_{s(y)}(-1)}f(z)dz=
-F_E(\tilde{\beta}_{s(y)}^m).
\]

By Theorem~\ref{int-iso}, we have
\[
 \sum_{k=0}^{\frac{N}{m}-1} F_E(\tilde{\beta}_{l}^m)
 = \sum_{l=0}^{\frac{N}{m}-1} \sE_E^0 \circ \{\tilde{\beta}_{l}^m(1), 
 \tilde{\beta}_{l}^m(-1)\}=-\int_{z_0}^{\tilde{\beta} z_0} k^*(\omega_E).
\]
By  the definition of the period $\pi_E$ of the Eisenstein series $E(z),$ we get
\[
 \int_{z_0}^{\gamma' z_0} k^*(\omega_E)=\int_{z_0}^{\gamma' z_0}
  [2 E(z)-\frac{1}{2} E(\frac{z+1}{2})] dz= 2 \pi_E(\gamma')-\pi_E(h\gamma'h^{-1}).
\]
As in the proof of Proposition 35, p.no 281,  \cite{MR3463038} (replacing p by m and q by $\frac{N}{m},$  we have $\pi_E(h \gamma'h^{-1})=\frac{N}{m}a_0(E[\frac{1}{m}])$ and $\pi_E(\gamma')
=2 \frac{N}{m} a_0(E[\frac{1}{m}])$ and
$ \int_{z_0}^{\gamma' z_0} k^*(\omega_E)=3 a_0(E[\frac{1}{m}])$.  
\[
F_E(I)=-F_E(\alpha_{N})=\int_{1}^{-1} [2 E(z)-\frac{1}{2} E(\frac{z+1}{2})] dz=-\int_{-1}^{\beta(-1)} [2 E(z)-\frac{1}{2} E(\frac{z+1}{2})] dz=-3 a_0(E), 
\]
we conclude that $\tilde{C}(\sE_E^0)=[F_E(I)-F_E(\alpha_{N})]=-6 a_0(E)$ and hence $\delta^0(\sE_E^0)=\delta(\sE_E^0)=-6 \delta(E)$. 
\end{proof}
The inclusion map $i : (X_0(N) - R \cup I, \partial(X_0(N)) \rightarrow (X_0(N), \partial(X_0 (N))$ induces an onto map $i_* :
\HH_1(X_0(N) - R \cup I, \partial(X_0(N), \Z) \rightarrow \HH_1(X_0(N), \partial(X_0(N)), \Z)$ with
 $i_{*}( [ g]^* ) = \xi (g).$ Note that 
$ \delta ( [g]^* ) = [g.0] - [ g.\infty ] = \delta^{'}( \xi(g) ) =  \delta^{'}( i_{*} ( [g]^* ) ).$ 
From [\S~\ref{relative homology}], we have $\delta(c) = \delta^{'}(i_{*}(c))$ for all homology class
$c \in  \HH_1 (X_0(N)-R \cup I, \partial(X_0(N), \Z)$. 
\section{Proof of Theorem ~\ref{Main-thm}}
\begin{proof}
By [\cite{MR1363498}, Cor.\,3], we obtain $i_*(\sE_E') \circ c = \sE_E' \circ i^* c = \int_c i_*(E(z)dz).$
Hence, $i_*(\sE'_E)$ is the Eisenstein element inside the space of modular symbols corresponding to $E$. 
By Proposition~\ref{boundact} and ~\ref{bound}, the boundary of $\pi_*(\sE^0_E)$ is same as the boundary of $6i_*(\sE'_E)$.
There is a non-degenerate bilinear pairing 
$ S_2(\Gamma_0(N)) \times \HH_1(X_0(N),\R) \rightarrow \C$
given by  $(f,c)=\int_c f(z) dz$. Hence, the integrals of the 
holomorphic differentials over $\HH_1(X_0(N),Z)$ are not always zero.
By  [\cite{MR1405312}, Lemma 5], the integrals of every holomorphic differentials over $i_*(\sE'_E)$ and $i_*(\pi_*(\sE_E^0))$ are always zero. 
\[
\mathrm{We} \ \mathrm{deduce} \ \mathrm{that} \
\sE_E=i_*(\sE'_E)=\frac{1}{6} i_*\pi_*(\sE_E^0)=\frac{1}{6} \sum_{g \in \sP^(\Z/N\Z)} F_E(g) \xi(g), \
\mathrm{for}  \ E \in \tE_{N}.
\]
Let $e_{N} \in  \HH_1 (X_0 (N), \Z) \otimes \R$ be the winding element.
Let ${1 < m < N, m \mid N}.$
The constant Fourier coefficients of $E_{N}$ at cusps $0$ and $\frac{1}{m},$ and $\infty$ are 
$\frac{1-N}{24N}$, $0$, $0$ and $\frac{N-1}{24}$ respectively [as in the proof of Lamma 38,   \cite{MR3463038} (replacing p by m and q by $\frac{N}{m}.$]
Hence, we obtain
 $$(1-N) e_{N}=\sum_{v \in(\Z/N\Z)^*} F_{N}((1,v)) \{0,\frac{1}{x}\}.$$
 \end{proof}
\section{Concluding Remarks}
Generalization of the results in this paper to any arbitrary level $N$ is an interesting question. The methods in this paper works only for squarefree level.
\begin{remark}
For the Eisenstein series $E_m$ $\in$ $E_2(\Gamma_0(m))$, $\frac{1}{m}$ represents the cusp $\infty$ and
$\frac{m}{N}$ represents the cusp $0.$ 
We deduce that $a_0(E_m[\beta_{0}])=\frac{m-1}{24}$ 
and $a_0(E_m[\gamma_{0}]) = \frac{1-m}{24m}.$
\end{remark}
\begin{remark}
For the Eisenstein series 
$E_{N},$ by [Lemma 4,~\cite{MR1405312}] the Eisenstein elements can be written explicitly if 
$g=(r-1,r+1)$ as follows.
\[
F_{N}((r-1,r+1))=\sum_{h=0}^{N-1} \overline{B_1} (\frac{hr}{2N}).
\]
\end{remark}

\section{Acknowledgements}
The author would like to thank Lo{\"{\i}}c Merel, Debargha Banerjee, Narasimha Kumar and Joseph Oesterl{\'e} for their helpful discussion, suggestions and comments. Part of this work was done at IMJ-PRG, France. The author would like to thank the mathematics department for the hospitality. The author would like to thank IISER, TVM for providing the excellent working conditions.
	

\bibliographystyle{unsrt}
\bibliography{Eisensteinquestion.bib}
\end{document}